\documentclass[10pt,twocolumn,twoside]{IEEEtran}
\usepackage[margin=.76in,footskip=0.74in]{geometry}
\usepackage{amssymb,amsfonts,amsmath,amsthm}
\usepackage{mathtools}
\usepackage{eucal,bm}
\usepackage{gensymb,subfigure}
\usepackage{multirow}
\usepackage{booktabs}
\usepackage{times,xcolor,url}
\usepackage{comment}
\usepackage{cite}
\usepackage{graphicx}
\usepackage[all]{xy}
\usepackage{algorithm}
\usepackage{algpseudocode}
\usepackage{subfigure}
\usepackage{bm}
\usepackage{color}
\usepackage{wrapfig}
\usepackage{epsf}
\usepackage{times,mathptm}
\usepackage{url}
\usepackage{nomencl}
\newtheorem{theorem}{Theorem}

\newtheorem{lemma}{Lemma}

\newcommand{\squeezeup}{\vspace{-2.5mm}}
\begin{document}
\title{Learning Topology of Distribution Grids using only Terminal Node Measurements}
\author{\IEEEauthorblockN{Deepjyoti~Deka\dag, Scott~Backhaus\dag, and Michael~Chertkov\dag\\}
\IEEEauthorblockA{\dag Los Alamos National Laboratory, USA\\
Email: deepjyoti@lanl.gov, backhaus@lanl.gov, chertkov@lanl.gov}}

\maketitle

\begin{abstract}
Distribution grids include medium and low voltage lines that are involved in the delivery of electricity from substation to end-users/loads. A distribution grid is operated in a radial/tree-like structure, determined by switching on or off lines from an underling loopy graph. Due to the presence of limited real-time measurements, the critical problem of fast estimation of the radial grid structure is not straightforward. This paper presents a new learning algorithm that uses measurements only at the terminal or leaf nodes in the distribution grid to estimate its radial structure. The algorithm is based on results involving voltages of node triplets that arise due to the radial structure. The polynomial computational complexity of the algorithm is presented along with a detailed analysis of its working. The most significant contribution of the approach is that it is able to learn the structure in certain cases where available measurements are confined to only half of the nodes. This represents learning under minimum permissible observability. Performance of the proposed approach in learning structure is demonstrated by experiments on test radial distribution grids.
\end{abstract}

\section{Introduction}
\label{sec:intro}
The power grid is operationally divided hierarchically into transmission and distribution grids. While the transmission grid connects the generators and includes high voltage lines, the distribution grid comprises of medium and low voltage lines that connect the distribution substation to the end users/loads. Aside from low voltages, distribution grids are structurally distinguished from the transmission side by their operational radial structure. A distribution grid is operated as a tree with a substation placed at the root node/bus that is connected via intermediate nodes and lines to the terminal end nodes/households. This radial operational structure is derived from an underlying loopy network by switching on and off some of the lines (network edges) \cite{distgridpart1}. The specific structure may be changed from one radial configuration to another, by reversing the switch statuses. Fig.~\ref{fig:picHinv} presents an illustrative example of a radial distribution grid derived from an underlying graph.

Historically, the distribution grid has had limited presence of real time devices on the lines, buses and feeders \cite{hoffman2006practical}. Due to this, real-time monitoring of the operating radial structure and the states of the resident buses on the distribution side is not straightforward. These estimation problems, previously neglected, have become of critical importance due to introduction of new devices like batteries and electric vehicles, and intermittent generation resources like rooftop solar panels. Optimal control in today's distribution grid requires fast topology and state estimation, often from limited real-time measurements. In this context, it needs to be mentioned that smart meters, micro-PMUs \cite{micropmu}, frequency measurements devices (FNETs)\cite{fnet} and advanced sensors (internet-of-things) capable of reporting real-time measurements are being deployed on the distribution side. However, in the current scenario, such devices are often limited to households/terminal nodes as their primary purpose for installation is services like price controllable demand and load monitoring. A majority of the intermediate lines and buses that connect the distribution substation to the terminal nodes do not have real-time measurement devices and are thus unobserved in terms of their structure and state. This hiders real-time topology and state estimation.

This paper is aimed at developing a learning framework that is able to overcome the lack of measurements at the intermediate nodes. Specifically, the primary goal of our work is to propose an algorithm to learn the operating radial topology using only real-time voltage magnitude measurements from the terminal nodes. The reliance only on measurements from terminal nodes is crucial as it makes our work applicable to realistic deployment of smart devices in today's distribution grids. Further, our learning algorithm is able to tolerate a much higher fraction of missing nodes (with unavailable data) compared to prior work in this area. Our approach is based on provable relations in \emph{voltage magnitudes of triplets and pairs of terminal nodes} that are used to discover the operational edges iteratively. Computationally, the algorithm has polynomial complexity in the number of nodes in the system.
\squeezeup
\subsection{Prior Work}
Topology learning in radial distribution grids has received attention in recent times. Learning techniques in this area vary, primarily based on the operating conditions and measurements available for estimation. The authors of \cite{he2011dependency} use a Markov random field model for nodal phases to identify faults in the grids. \cite{distgrid_pscc} uses conditional independence tests to identify the grid topology in radial grids. \cite{bolognani2013identification} uses signs of elements in the inverse covariance matrix of nodal voltages to learn the operational topology. Signature/enevelope based identification of topology changes is proposed in \cite{berkeley}. Such comparison based schemes are used for parameter estimation in \cite{sandia1, sandia2}. In contrast with nodal voltage measurements, line flow measurements are used in a maximum likelihood based scheme for topology estimation in \cite{ramstanford}. In previous work \cite{distgridpart1,distgridpart2}, authors have analyzed topology learning schemes that rely on trends in second moments of nodal voltage magnitudes. Further, a spanning tree based topology learning algorithm is proposed in \cite{distgrid_ecc}. This line of work \cite{distgridpart1, distgridpart2,distgrid_ecc} is close in spirit to machine learning schemes \cite{ravikumar2010high,anandkumar2011high} developed to learn the structure of general probabilistic graphical models \cite{wainwright2008graphical}. A major limitation of the cited literature is that they assume measurement collection at most, if not all, nodes in the system. To the best of our knowledge, work that discuss learning in the presence of missing/unobserved nodes \cite{distgridpart2, distgrid_ecc} assume missing nodes to be separated by greater than two hops in the grid. As discussed earlier, distribution grids often have real-time meters only at terminal nodes and none at intermediate nodes that may be adjacent (one hop away).

\subsection{Contribution of This Work}
In this paper, we discuss topology learning in the radial distribution grid when measurements are limited to real-time voltage readings at the terminal nodes (end-users) alone. All intermediate nodes are unobserved and hence assumed to be missing nodes. We analyze voltages in the distribution grid using a linear lossless AC power flow model \cite{distgridpart1,distgridpart2} that is analogous to the popular \cite{89BWa,89BWb} LinDistFlow equations. For uncorrelated fluctuations of nodal power consumption, we construct functions of voltage magnitudes at a pair or triple of terminal nodes such that their values depend on the edges that connect the group of nodes. These functions provide the necessary tool to develop our learning scheme that iteratively identifies operational edges from the leaf onward to the substation root node. We discuss the computational complexity of the learning scheme and show that it is a third order polynomial in the number of nodes. In comparison to existing work, our approach is able to learn the topology and thereby estimate usage statistics in the presence of much greater fraction of missing/unobserved nodes. In fact, in limiting configurations, our learning algorithm is able to determine the true structure and statistics even when half of the nodes are unobserved/missing. We demonstrate the performance of our algorithm through experiments on test distribution grids.

The rest of the manuscript is organized as follows. Section \ref{sec:structure} introduces notations used in the paper and describes the grid topology and power flow models used for analysis in later sections. Section \ref{sec:trends} mentions assumptions made and describes important properties involving voltage measurements at terminal nodes. Our algorithm to learn the operating radial structure of the grid is discussed in Section \ref{sec:algo1} with detailed examples. We detail the computational complexity of the algorithm in Section \ref{sec:complexity}. Simulation results of our learning algorithm on test radial networks are presented in Section \ref{sec:experiments}. Finally, Section \ref{sec:conclusions} contains conclusions and discussion of future work.
\section{Distribution Grid: Structure and Power Flows}
\label{sec:structure}
\textbf{Radial Structure}: We denote the underlying graph of the distribution grid by the graph ${\cal G}=({\cal V},{\cal E})$, where ${\cal V}$ is the set of buses/nodes and ${\cal E}$ is the set of all undirected lines/edges. We term nodes by alphabets ($a$, $b$,...). An edge between nodes $a$ and $b$ is denoted by $(ab)$. An illustrative example is given in Fig.~\ref{fig:picHinv}. The operational grid consisting of one tree $\cal T$ with nodes ${\cal V}^{\cal T}$ and operational edge set ${\cal E}^{\cal T} \subset {\cal E}$ as shown in Fig.~\ref{fig:picHinv}. Our results are immediately extended to the case $K>1$. In tree $\cal T$, ${\cal P}_{a}$ denote the set of edges in the unique path from node $a$ to the root node (reference bus). A node $c$ is called a \emph{descendant} of node $a$ if the path from node $c$ to the root passes through $a$, i.e., ${\cal P}_{c} \subset {\cal P}_{a}$. We use ${\cal D}_{a}$ to denote the set of descendants of $a$ and include node $a$ in ${\cal D}_{a}$ by definition. If edge $(ac) \in {\cal E}^{\cal T}$ and $c$ is a descendant of $a$, we term $a$ as \emph{parent} and $c$ as its \emph{child node}. Further, as discussed in the Introduction, \emph{terminal nodes/leaf nodes} in the tree represent end-users or households that are assumed to be equipped with real-time nodal meters. The remaining intermediate nodes (not terminal nodes) are not observed and hence termed as \emph{missing nodes}.  These definitions are illustrated in Fig \ref{fig:picHinv}. Next we describe the notation used in the power flow equations.
\begin{figure}[!bt]
\centering
\includegraphics[width=0.29\textwidth,height=.25\textwidth]{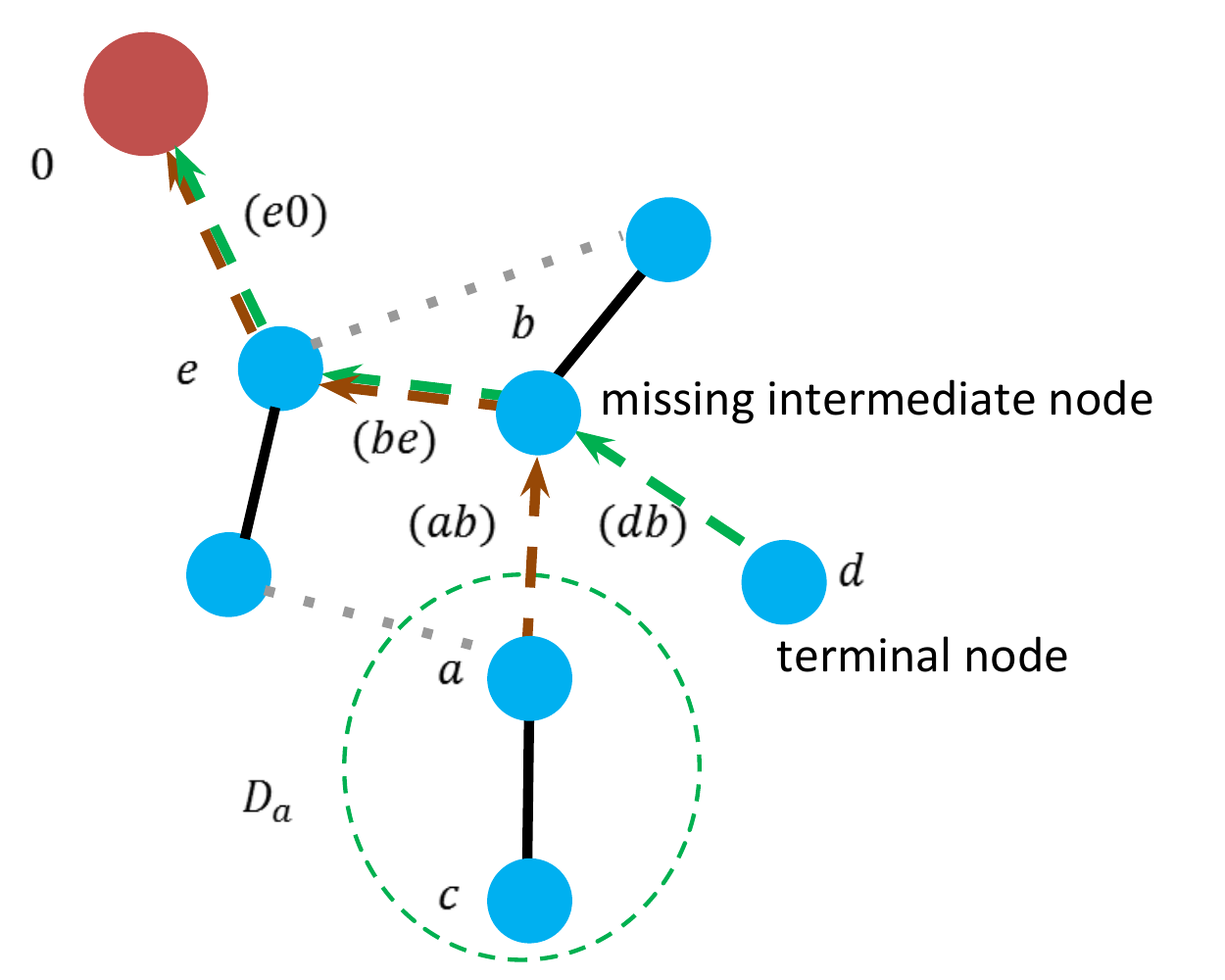}
\squeezeup
\caption{Radial distribution grid tree with substation/root node colored in red. Dotted grey lines represent open switches. Terminal leaf nodes ($d$) represent end-users from which real-time measurements are collected. The intermediate missing nodes ($b$) are unobserved. Here, nodes $a$ and $c$ are descendants of node $a$. Dashed lines represent the paths from nodes $a$ and $d$ to the root node.
\label{fig:picHinv}}
\end{figure}

\textbf{Power Flow Models}: We use $z_{ab}=r_{ab}+i x_{ab}$ to denote the complex impedances of line $(ab)$ ($i^2=-1$) where $r_{ab}$ and $x_{ab}$ represent the line resistance and reactance respectively. Let real valued scalars, $v_a$, $\theta_a$, $p_a$ and $q_a$ denote the voltage magnitude, voltage phase, active and reactive power injection respectively at node $a$. At each node $a$, Kirchhoff's law of power flow relate them as follows:
\squeezeup
\begin{align}
\squeezeup
\squeezeup
P_a =p_a+i q_a
= \underset{b:(ab)\in{\cal E}^{\cal T}}{\sum}\frac{v_a^2-v_a v_b\exp(i\theta_a-i\theta_b)}{z_{ab}^*}\label{P-complex1}
\squeezeup
\end{align}
Note that Eq.~(\ref{P-complex1}) is nonlinear and non-convex. Under realistic assumption that losses of both active and reactive power losses on each line of tree ${\cal T}$ is small, we can neglect second order terms in Eq.~(\ref{P-complex1}) to achieve the following linearized form \cite{distgridpart1,distgridpart2,bolognani2013identification}:
\squeezeup
\begin{align}
\squeezeup
p_a&=\underset{b:(ab)\in{\cal E}^{{\cal T}}}{\sum}\left(\beta_{ab}(\theta_a-\theta_b)+
g_{ab}(v_a-v_b)\right),\label{PF_LPV_p}\\
q_a&=\underset{b:(ab)\in{\cal E}^{{\cal T}}}{\sum}\left(-g_{ab}(\theta_a-\theta_b)+
\beta_{ab}(v_a-v_b)\right)\label{PF_LPV_q}\\
\text{where~} g_{ab}&\doteq\frac{r_{ab}}{x_{ab}^2+r_{ab}^2}, \beta_{ab}\doteq\frac{x_{ab}}{x_{ab}^2+r_{ab}^2} \label{LCPF_p}
\squeezeup
\end{align}
As shown in \cite{distgridpart1}, Eqs.~(\ref{PF_LPV_p}),(\ref{PF_LPV_q}) are equivalent to the LinDistFlow equations for power flow in distribution grid \cite{89BWa,89BWb,89BWc}, when deviations in voltage magnitude are assumed to be small. Similar to LinDistFlow model, Eqs.~(\ref{PF_LPV_p}),(\ref{PF_LPV_q}) are lossless with sum power equal to zero ($\sum_{a \in {\cal V}^{\cal T}}P_a = 0$). Further, note that both active and reactive power injections are functions of difference in voltage magnitudes and phases of neighboring nodes. Thus, the analysis of the system can be reduced by one node termed as reference node with voltage magnitude and phase at all other nodes being measured relative to this node. In our case, we take the substation/root node as the reference node with voltage magnitude $1$ and phase $0$ respectively. The reference node's injection also balances the power injections in the remaining network. Inverting Eqs.~(\ref{PF_LPV_p}),(\ref{PF_LPV_q}) for the reduced system (without the reference node), we express voltages as a function of nodal power injections in the following vector form:
\squeezeup
\begin{align}
v = H^{-1}_{1/r}p + H^{-1}_{1/x}q~~ \theta = H^{-1}_{1/x}p - H^{-1}_{1/r}q  \label{LC_PF}
\end{align}
We term this as the \textbf{Linear Coupled Power Flow (LC-PF) model} where $p$, $q$, $v$ and $\theta$ are the vectors of real power, reactive power injections, voltage magnitudes and phase angles respectively at the non-substation nodes. $H_{1/r}$ and $H_{1/x}$ are the reduced weighted Laplacian matrices for tree $\cal T$ where  reciprocal of resistances and reactances are used respectively as edge weights. The reduction is achieved by removing the row and column corresponding to the reference bus in the original weighted Laplacian matrix. We denote the mean of a random vector $X$ by $\mu_{X} = \mathbb{E}[X]$. For two random vectors $X$ and $Y$, the covariance matrix is denoted by $\Omega_{XY} = \mathbb{E}[(X-\mu_{X})(Y-\mu_{Y})^T]$. Using Eq.~(\ref{LC_PF}), we can relate the means and covariances of voltage magnitudes with those of active and reactive injection as follows:
\begin{align}
\mu_v &= H^{-1}_{1/r}\mu_p + H^{-1}_{1/x}\mu_q\label{means}\\
\Omega_{v} &= H^{-1}_{1/r}\Omega_{p}H^{-1}_{1/r} + H^{-1}_{1/x}\Omega_qH^{-1}_{1/x}+H^{-1}_{1/r}\Omega_{pq}H^{-1}_{1/x}\nonumber\\
&~+H^{-1}_{1/x}\Omega_{qp}H^{-1}_{1/r}\label{volcovar1}
\end{align}
Using these statistical quantities, we discuss useful identities that arise in radial distribution grids in the next section that form the basis of our learning algorithms.

\section{Properties of Voltage Magnitudes in Radial Grids}
\label{sec:trends}
First, we make the following assumption regarding statistics of power injections at the non-substation grid nodes, under which our results hold.

\textbf{Assumption $1$:} Fluctuations of active and reactive powers at different nodes are uncorrelated. Thus, $\forall a \neq b$ non-substation nodes, $\Omega_p(a,b) = \Omega_q(a,b)= \Omega_{qp}(a,b) = 0$.

As considered in prior literature \cite{bolognani2013identification,distgridpart2}, this assumption is reasonable over short time-intervals where fluctuations in nodal power usage at households/terminal nodes are independent and hence uncorrelated. Intermediate nodes that do not represent end-users are uncorrelated if they have independent usage patterns. Specifically, for intermediate nodes involved in separation of power into downstream lines and without any major nodal usage, the net power injection is contributed by leakage or device losses and hence uncorrelated from the rest. Note that Assumption $1$ does not specify the class of distributions that can model individual node's power injection. It is applicable when nodal injections are negative (loads), positive (due to local generation) or a mixture of both. In future work, we will relax this assumption and discuss learning in the presence of positively correlated end-user injection profiles.

Next, we mention an analytical statement relating the inverse of the reduced Laplacian matrix for a radial graph that we use in our later results.
\begin{lemma}\cite{68Resh,distgridpart1}\label{lemma1}
The reduced weighted Laplacian matrix $H_{1/r}$ for tree $\cal T$ satisfies
\begin{align}
 H_{1/r}^{-1}(a,b)&= \sum_{(cd) \in {\cal P}_a\bigcap {\cal P}_b} r_{cd} \label{Hrxinv}
\end{align}
\end{lemma}
In other words, the $(a,b)^{th}$ entry in $H^{-1}_{1/r}$ is equal to the sum of line resistances of edges common to paths from node $a$ and $b$ to the root. For example, in Fig.~\ref{fig:picHinv}, $H_{1/r}^{-1}(a,d) = r_{be}+ r_{e0}$. Using Eq.~(\ref{Hrxinv}) it follows immediately that if node $b$ is the parent of node $a$, then $\forall c$
\begin{align}
{\huge H}_{1/r}^{-1}(a,c)-{\huge H}_{1/r}^{-1}(b,c) &&=\begin{cases}r_{ab} & \quad\text{if node $c \in {\cal D}_a$}\\
0 & \quad\text{otherwise,} \end{cases} \label{Hdiff}
\end{align}
Next, consider the function $\phi$ defined over two nodes $a$ and $b$ as $\phi_{ab} = \mathbb{E}[(v_a - \mu_{v_a})-(v_b-\mu_{v_b})]^2$. $\phi_{ab}$ represents \textit{the variance of the difference in voltage magnitudes at nodes $a$ and $b$}. Using Eq.~(\ref{volcovar1}) we can write $\phi_{ab}$ in terms of power injection statistics in tree $\cal T$ as follows
\begin{small}
\begin{align}
&\phi_{ab} = \Omega_{v}(a,a) - 2\Omega_{v}(a,b) + \Omega_{v}(b,b) \label{expand}\\
&= \smashoperator[lr]{\sum_{d \in {\cal T}}}(H^{-1}_{1/r}(a,d)- H^{-1}_{1/r}(b,d))^2\Omega_p(d,d)\nonumber\\
&+(H^{-1}_{1/x}(a,d)- H^{-1}_{1/x}(b,d))^2 \Omega_q(d,d)\nonumber\\
&+2\left(H^{-1}_{1/r}(a,d)- H^{-1}_{1/r}(b,d)\right)\left(H^{-1}_{1/x}(a,d)- H^{-1}_{1/x}(b,d)\right)\Omega_{pq}(d,d)\label{usediff_1}
\end{align}
\end{small}
Note that Lemma \ref{lemma1} and Eq.~(\ref{Hdiff}) can be inserted in Eq.~(\ref{usediff_1}) to simply it. In fact, doing so lets us derive properties of $\phi$ for terminal nodes in tree $\cal T$ as discussed next.

\begin{theorem} \label{theoremcase1}
Let node $b$ be the parent of nodes $a$ and $c$ in $\cal T$ such that $(ab)$ and $(bc)$ are operational edges (see Fig.~\ref{fig:item2}). Then
\begin{align}
\phi_{ac} &= \smashoperator[lr]{\sum_{d \in {\cal D}_a}}r_{ab}^2\Omega_p(d,d)+x_{ab}^2 \Omega_q(d,d)+2r_{ab}x_{ab}\Omega_{pq}(d,d)\nonumber\\
&+ \smashoperator[lr]{\sum_{d \in {\cal D}_c}}r_{bc}^2\Omega_p(d,d)+x_{bc}^2 \Omega_q(d,d)+2r_{bc}x_{bc}\Omega_{pq}(d,d)\label{equal1}\\
&= \phi_{ab} + \phi_{bc}  \label{equal2}
\end{align}
\end{theorem}
\begin{proof}
Observe Lemma \ref{lemma1}. As $(ab)$ and $(bc)$ are operational edges, the only nodes $d$ such that $(H^{-1}_{1/r}(a,d)- H^{-1}_{1/r}(c,d)) \neq 0$ are either descendants of $a$ (set ${\cal D}_a$) or of $c$ (set ${\cal D}_c$). Further ${\cal D}_c$ and ${\cal D}_a$ are disjoint. When $d \in {\cal D}_a$, $(H^{-1}_{1/r}(a,d)- H^{-1}_{1/r}(c,d)) = r_{ab}$, while when $d \in {\cal D}_c$, $(H^{-1}_{1/r}(a,d)- H^{-1}_{1/r}(c,d)) = -r_{bc}$. Using this in the formula for $\phi_{ac}$ in Eq.~(\ref{usediff_1}) gives us the relation. The equality $\phi_{ac} = \phi_{ab} + \phi_{bc}$ is verified by plugging values in Eq.~(\ref{usediff_1}) for $\phi_{ab}$, $\phi_{bc}$ and $\phi_{ac}$.
\end{proof}
\begin{figure}[!bt]
\centering
\hspace*{\fill}
\subfigure[]{\includegraphics[width=0.16\textwidth]{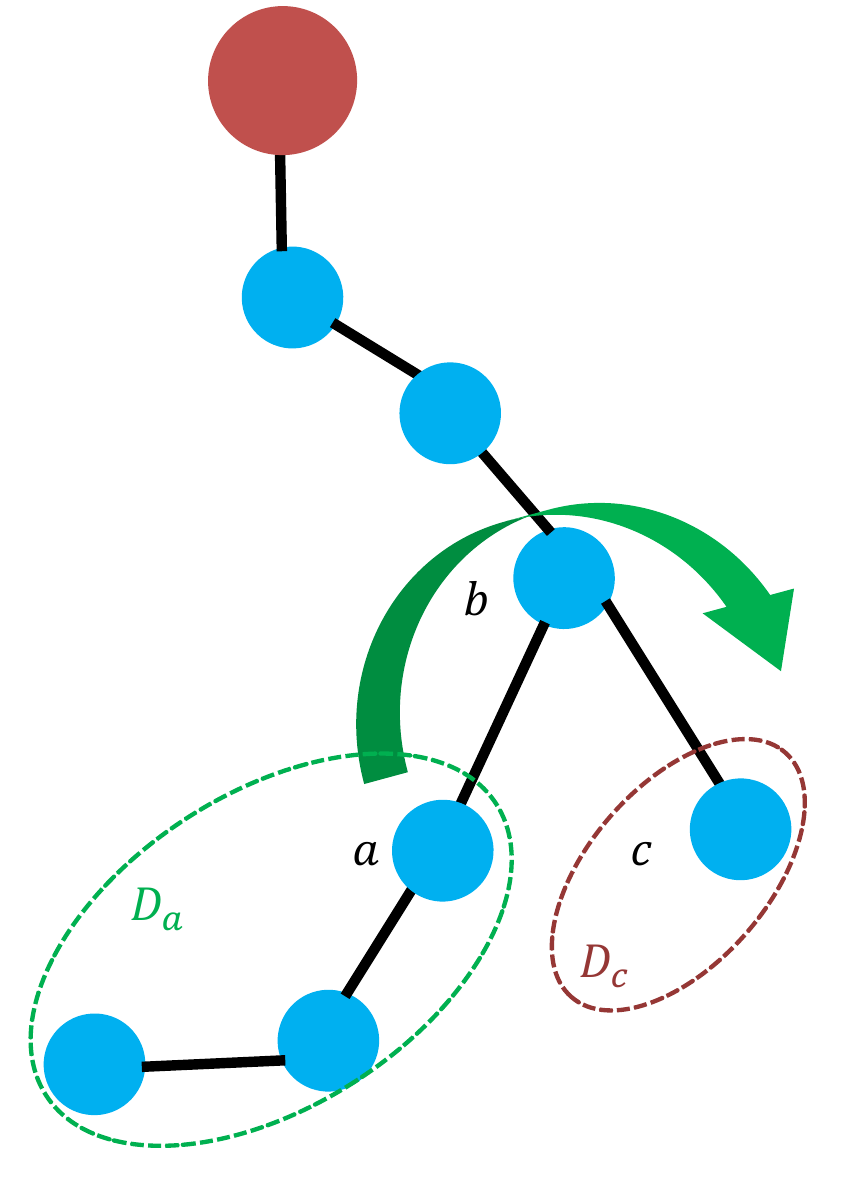}\label{fig:item2}}\hfill
\subfigure[]{\includegraphics[width=0.16\textwidth]{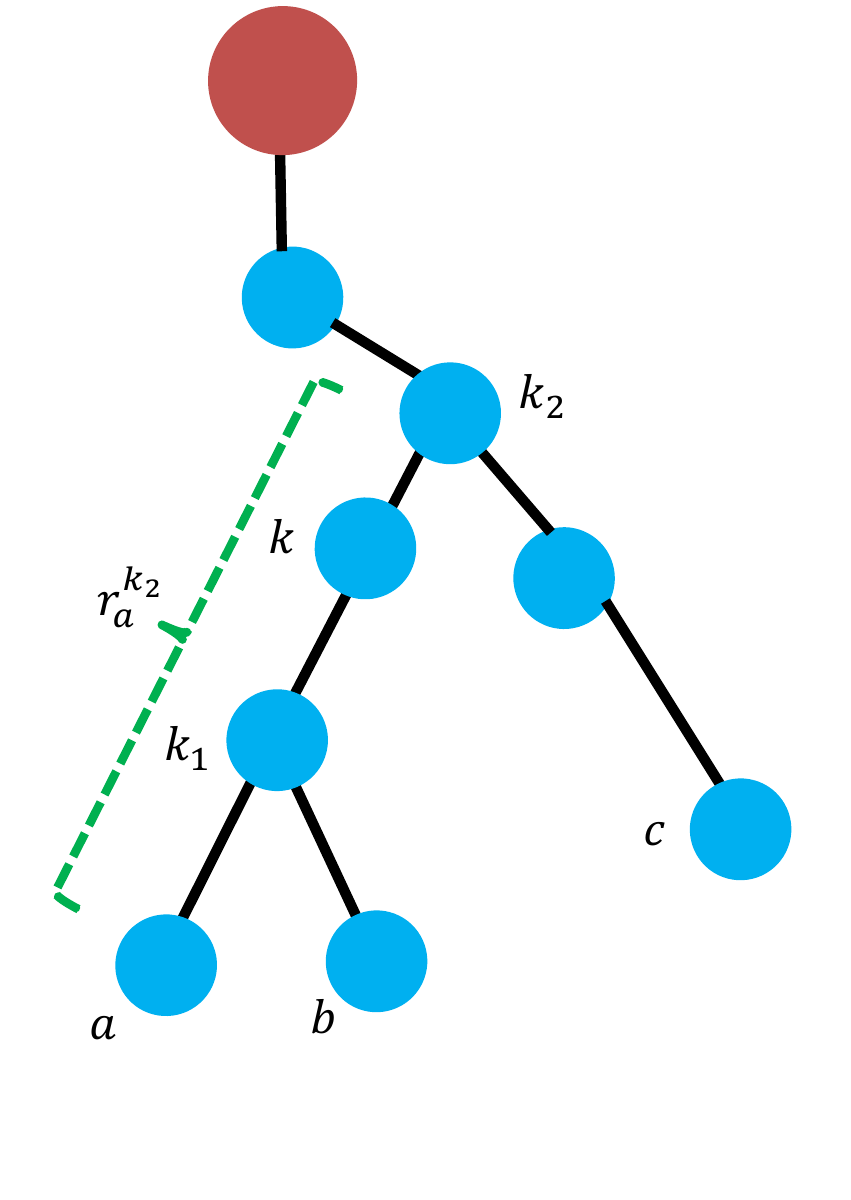}\label{fig:item3}}
\hspace*{\fill}
\squeezeup
\caption{(a) Distribution grid tree with nodes $a$ and $c$ as children of common parent node $b$. (b)Terminal nodes $a$ and $b$ have common parent $k_1$. $k_1$ and terminal node $c$ are descendants of $k_2$. $r^{k_2}_a$ is the sum of resistance on lines $(ak_1), (k_1k)$ and $(kk_2)$ connecting nodes $a$ and $k$.
\label{fig:items}}
\end{figure}
A few points are in order. First, note that the only operational lines whose impedances appear on the right side of Eq.~(\ref{equal1}) are $(ab)$ and $(bc)$. For the special case where $a$ and $c$ are terminal nodes with parent $b$, the relation reduces to the following:
\begin{align}
\phi_{ac} &= r_{ab}^2\Omega_p(a,a)+x_{ab}^2 \Omega_q(a,a)+2r_{ab}x_{ab}\Omega_{pq}(a,a)\nonumber\\
&+ r_{bc}^2\Omega_p(c,c)+x_{bc}^2 \Omega_q(c,c)+2r_{bc}x_{bc}\Omega_{pq}(c,c) \label{equal3}
\end{align}
If the covariance of injections at $a$ and $b$ are known, the above condition can be checked for each potential parent node `$b$' in linear time to identify the true parent. Second, the equality in Eq.~(\ref{equal2}) is true only when ${\cal P}_{a} \cap {\cal P}_{c} = {\cal P}_{b}$. It is replaced by a strict inequality for other configurations that we omit discussing as they are outside the scope of our learning conditions. The next theorem gives a result that relates the voltages at three terminal nodes.
\begin{theorem} \label{theoremcase2}
Let terminal nodes $a$ and $b$ have common parent node $k_1$. Let $c$ be another terminal node such that $c, k_1 \in {\cal D}_{k_2}$ and ${\cal P}_{k_1} \cap {\cal P}_{c} = {\cal P}_{k_2}$ for some intermediate node $k_2$ (see Fig.~\ref{fig:item3}). Let $r^{k_2}_a$ and $x^{k_2}_a$ denote the sum of resistance and reactance respectively on lines on the path from node $a$ to node $k_2$, i.e., $r^{k_2}_a =\smashoperator[lr]{\sum_{(ef) \in {\cal P}_{a} - {\cal P}_{k_2}}}r_{ef}$, $x^{k_2}_a =\smashoperator[lr]{\sum_{(ef) \in {\cal P}_{a} - {\cal P}_{k_2}}}x_{ef}$. Define $r^{k_2}_b$, $r^{k_2}_{k_1}$ etc. in the same way. Then
\begin{small}
\begin{align}
\phi_{ac}-\phi_{bc} &= \Omega_p(a,a)((r^{k_2}_a)^2 -(r^{k_2}_{k_1})^2) + \Omega_q(a,a)((x^{k_2}_a)^2 -(x^{k_2}_{k_1})^2) \nonumber\\
&+\Omega_{pq}(a,a)(r^{k_2}_ax^{k_2}_a -r^{k_2}_{k_1}x^{k_2}_{k_1})-\Omega_p(b,b)((r^{k_2}_b)^2 -(r^{k_2}_{k_1})^2)\nonumber\\
&+ \Omega_q(b,b)((x^{k_2}_b)^2 -(x^{k_2}_{k_1})^2) + \Omega_{pq}(b,b)(r^{k_2}_bx^{k_2}_b -r^{k_2}_{k_1}x^{k_2}_{k_1})\label{equaltriple}
\end{align}
\end{small}
\end{theorem}
\begin{proof}
As $a$ and $b$ are terminal nodes with same parent $k_1$, for each node $d \neq a \neq b$, the paths from $d$, $a$ and $b$ to the root follow  ${\cal P}_{a} \cap {\cal P}_{d} = {\cal P}_{b} \cap {\cal P}_{d}$. Using Lemma \ref{lemma1}, we thus have $\forall d \neq a,b $, $H^{-1}_{1/r}(a,d) = H^{-1}_{1/r}(b,d)$ and 
\begin{small}
\begin{align}
&H^{-1}_{1/r}(a,d) - H^{-1}_{1/r}(c,d) = H^{-1}_{1/r}(b,d) - H^{-1}_{1/r}(c,d) \label{use1}\\
&H^{-1}_{1/r}(a,a) - H^{-1}_{1/r}(a,b) = r_{ak_1}, H^{-1}_{1/r}(b,b) - H^{-1}_{1/r}(b,a) = r_{bk_1}\label{use2}
\end{align}
\end{small}
As $k_1$ and $c$ are descendants of node $k_2$ and ${\cal P}_{k_1} \cap {\cal P}_{c} = {\cal P}_{k_2}$,
\begin{small}
\begin{align}
H^{-1}_{1/r}(a,a) - H^{-1}_{1/r}(c,a) = r^{k_2}_a, H^{-1}_{1/r}(b,b) - H^{-1}_{1/r}(c,b) = r^{k_2}_b \label{use3}
\end{align}
\end{small}
Further, using Eqs.~(\ref{use2}, \ref{use3}), we get
\begin{small}
\begin{align}
H^{-1}_{1/r}(b,a) - H^{-1}_{1/r}(c,a)= H^{-1}_{1/r}(a,b) - H^{-1}_{1/r}(c,b)= r^{k_2}_{k_1} \label{use4}
\end{align}
\end{small}
where $r^{k_2}_a, r^{k_2}_{k_1}$ etc. are defined in the statement of the theorem. Similar relations can be written for $H^{-1}_{1/x}$ terms as well. We now expand $\phi_{ac}$ and $\phi_{bc}$ using Eq.~(\ref{usediff_1}). Using Eq.~(\ref{use1},\ref{use3},\ref{use4}) in the expression for $\phi_{ac}-\phi_{bc}$ gives Eq.~(\ref{equaltriple}).
\end{proof}
Observe that the lines whose impedances appear on the right side of  Eq.~(\ref{equaltriple}) are $(ak_1), (bk_1)$ and the ones on the path from node $k_1$ to $k_2$. If this path is known till the penultimate node $k$ before $k_2$ (see Fig.~\ref{fig:item3}), then Eq.~(\ref{equaltriple}) can be used to learn edge $(kk_2)$ through a linear search among candidate nodes for $k_2$. In the next section, we discuss the use of this relation to learn the path from terminal pairs with common parent (here $a$ and $b$) to the root iteratively. Further, note that the value of $\phi_{ac} -\phi_{bc}$ is independent of injections at intermediate nodes and at terminal node $c$ as long as $c$ is a descendant of $k_2$. Thus, Eq.~(\ref{equaltriple}) only helps identify $c$'s relative position in the graph with respect to $k_2$. As shown in the next section, we are able to locate $c$'s parent using a post-order node traversal \cite{Cormen2001} in the grid graph.

\section{Algorithm to learn topology using terminal node measurements}
\label{sec:algo1}
Using the relations described in the previous section, we discuss our algorithm to learn the operational topology of the radial distribution grid. As mentioned in the Introduction, voltage measurements are available only at the terminal/leaf nodes. The observer also has access to power injection statistics (variances) at the terminal nodes. These statistics are either known from historical data or computed empirically from power injection measurements collected at the terminal nodes.  All intermediate nodes are missing and their voltage measurements and injection statistics are not observed/available. Further, we assume that impedances of all lines (open or operational) in $\cal E$ for the underlying loopy graph $\cal G$ are available as input. The task of the learning algorithm is to identify the set ${\cal E}^{\cal T}$ of operational edges in the radial grid $\cal T$. The substation root node in tree $\cal T$ is assumed to be connected to a single intermediate node. Otherwise each sub-tree connected to the substation node can be considered as a disjoint tree. First, we make the following restriction regarding the degree of missing intermediate nodes.

\textbf{Assumption 2:} All missing intermediate nodes are assumed to have a degree greater than $2$.

This assumption is necessary as without it, the solution to the topology learning problem will not be unique. As an example, consider the test distribution grid given in Fig.~\ref{fig:testline} where the path from leaf node $a$ to node $d$ passes through two intermediate nodes $b$ and $c$ of degree $2$ each. Both configuration $A$ (operations edges $(ab),(bc),(cd)$) and configuration $B$ (operations edges $(ac),(cb),(bd)$) are feasible operational topologies if voltage and power injection measurements are available only at the terminal node $a$. In other words, different values of injection covariances at nodes $b$ and $c$ can be taken such that relations between voltage and injections are satisfied in either configuration. Similar assumptions for uniqueness in learning general graphical models are mentioned in \cite{pearl}. For radial configurations that respect Assumption $2$, Algorithm $1$ learns the topology using measurements of voltage magnitudes and information of injection statistics at terminal nodes.
\begin{figure}[!bt]
\centering
\includegraphics[width=0.28\textwidth]{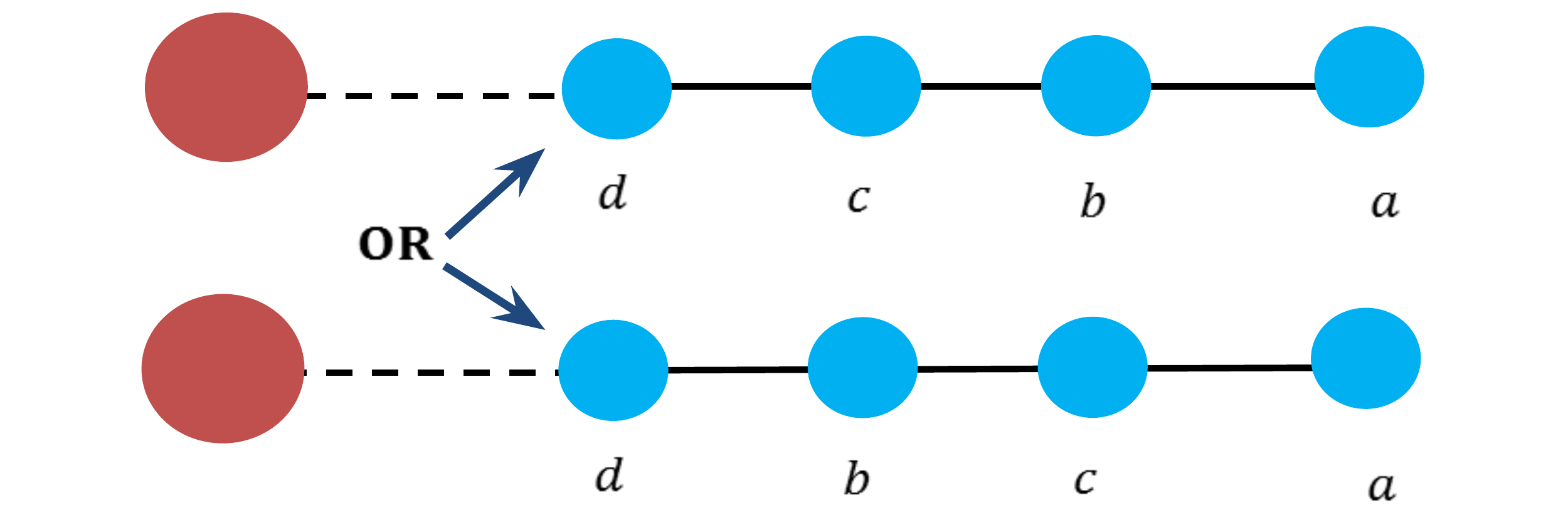}
\squeezeup
\caption{Distribution grid tree where terminal node $a$ is connected to node $d$ via two unknown intermediate nodes of degree $2$. Either configuration $d,c,b,a$ or $d,b,c,a$ for the intermediate nodes is a feasible structure given no measurements at nodes $b$ and $c$.
\label{fig:testline}}
\end{figure}
\begin{algorithm}
\caption{Topology Learning using Terminal Node Data}
\textbf{Input:} Injection covariances $\Omega_p, \Omega_q, \Omega_{pq}$ at terminal nodes $\cal L$, Missing node set ${\cal M} = {\cal V}^{\cal T} - {\cal L}$, $m$ voltage magnitude observations $v$ for nodes in ${\cal L}$, set of all edges $\cal E$ with line impedances.\\
\textbf{Output:} Operational Edge set ${\cal E}^{\cal T}$.
\begin{algorithmic}[1]
\begin{small}
\State $\forall$ nodes $a,c \in \cal L$ , compute $\phi_{ac} = \mathbb{E}[(v_a-\mu_{v_a}) -(v_c-\mu_{v_c})]^2$
\State $\forall a \in {\cal V}^{\cal T}$, define $par_{a} \gets \Phi$, $des_{a} \gets \Phi$
\ForAll{$a \in {\cal L}$} \label{stepa}
\If {$par_a = \Phi \& \exists c \in {\cal L}, b \in {\cal M}$ s.t. $\phi_{ab},c$ satisfy Eq.~(\ref{equal3})}
\State ${\cal E}^{\cal T} \gets {\cal E}^{\cal T} \cup \{(ab), (bc)\}$
\State $par_a \gets b, par_c \gets b, des_b \gets {a,c} $ \label{stepaend}
\EndIf
\EndFor
\State ${\cal L} \gets \{a: a \in {\cal L}, par_a = \Phi\}$, $tp \gets 1,{\cal M}_1 \gets \Phi$
\While{$tp >0$} \label{stepb}
\State ${\cal M}_2 \gets \{k_1: k_1 \in {\cal M}-{\cal M}_1, par_{k_1} = \Phi , des_{k_1} \neq \Phi\}$
\State ${\cal M}_1 \gets \{k_1: k_1 \in {\cal M}, par_{k_1} = \Phi , des_{k_1} \neq \Phi\}$
\ForAll{$k \in {\cal M}_1$ with $a,b \in des_{k}$}
\State $k_1\gets par_a$
\If{$\exists k_2 \in {\cal M}_1$ with ${\cal M}_2\cap \{k,k_2\} \neq \Phi$ with $c \in des_{k_2}$, s.t. $\phi_{ac}-\phi_{bc}$ satisfy Eq.~(\ref{equaltriple})}\label{stepb1}
\State ${\cal E}^{\cal T} \gets {\cal E}^{\cal T} \cup \{(kk_2)\}, par_k \gets k_2$
\Else
\If{$k \in {\cal M}_1, \exists k_2 \in {\cal M} -{\cal M}_1$, $c \in {\cal L}$, s.t. $\phi_{ac}-\phi_{bc}$ satisfy Eq.~(\ref{equaltriple})} \label{stepb2}
\State ${\cal E}^{\cal T} \gets {\cal E}^{\cal T} \cup \{(kk_2)\}, par_k \gets k_2, des_{k_2} \gets des_k$
\EndIf
\EndIf
\EndFor
\State $tp \gets |\{k_1: k_1 \in {\cal M}_1, par_{k_1} \neq \Phi\}|$
\EndWhile \label{stepbend}
\If $|{\cal M}_1| = 1$
\State Join $k \in {\cal M}_1$ to root
\EndIf
\State Form a post-order traversal node set ${\cal W}$ using $par_a$ for ${a:\forall a \in {\cal M}, des_a \neq \Phi}$ \label{postorder}
\ForAll{$c \in {\cal L}$} \label{stepc}
\For{$j \gets 1$ to $|{\cal W}|$}
\State $k_2 \gets {\cal W}(j)$ with $a,b \in des_{k_2}, k_1 \gets par_a$
\If{$\phi_{ac}-\phi_{bc}$ satisfy Eq.~(\ref{equaltriple})} \label{checkstepc}
\State ${\cal E}^{\cal T} \gets {\cal E}^{\cal T} \cup \{(ck_2)\}, {\cal W} \gets {\cal W}-\{k_2\}, j \gets |{\cal W}|$
\EndIf
\EndFor
\EndFor \label{stepcend}
\end{small}
\end{algorithmic}
\end{algorithm}

\textbf{Working}: Algorithm $1$ learns the structure of operational tree $\cal T$ iteratively from leaf node pairs up to the root. Set $\cal L$ represents the current set of terminal nodes/leaves with unknown parents, while set $\cal M$ represents the set of all missing intermediate nodes. Note that for each node $a$, $par_a$ denotes its identified parent while $des_{a}$ denotes a set of two leaf nodes that are its descendants. $\Phi$ represents the empty set. The two sets ($par$ and $des$) are used to keep track of the current stage of the learnt graph.  There are three major stages in the edge learning process in Algorithm $1$.

\textbf{First}, we build edges between leaf nodes pairs $a,c \in {\cal L}$ to their common missing parent node $b$ in Steps \ref{stepa} to \ref{stepaend}. Here the relation for $\phi_{ab}$ derived in Theorem \ref{theoremcase1} and Eq.~(\ref{equal3}) is used to identify the true missing parent in set $\cal M$. Note that only parents with two or more leaf nodes as children can be identified at this stage. Parents that have at most one leaf node as child (other children being missing intermediate nodes) are not identified by this check.

\textbf{Second}, we remove nodes with discovered parents from the set of leaf nodes $\cal L$ then iteratively learn the edges between intermediate nodes in Steps \ref{stepb} to \ref{stepbend}. At each stage of the iteration, ${\cal M}_2$ denotes the set of intermediate nodes with unknown parents whose descendants ($des$) were discovered in the previous iteration, while ${\cal M}_1$ denotes the set of all intermediate nodes with unknown parent ($par$) and known descendants. Clearly ${\cal M}_2 \subset {\cal M}_1$. The parent $k_2$ of each node $k$ in ${\cal M}_2$ is identified by using relation Eq.~(\ref{equaltriple}) in Theorem \ref{theoremcase2} for its descendants $a$ and $b$. There are two cases considered for candidate parent $k_2$: (A) $k_2$ belongs to ${\cal M}_1$ and has known descendant $c$ (Step \ref{stepb1}), and (B) $k_2$ belongs to ${\cal M} -{\cal M}_1$ and has no discovered descendant (Step \ref{stepb2}). For the second case, the algorithm tries to find some leaf node $c$ that is potentially an undiscovered descendant of $k_2$ by checking for the relation $\phi_{ac}-\phi_{bc}$  specified in Eq.~(\ref{equaltriple}). If the relation holds, node $k_2$ is denoted as the parent of $k_1$. As it is not clear if leaf $c$ is an immediate descendant (child) of $k_2$, no edge to $c$ is added. The iterations culminate when no additional edge between intermediate nodes is discovered.

\textbf{Third}, in Steps \ref{stepc} to \ref{stepcend}, the parents of the remaining leaf nodes in $\cal L$ are identified. Note that Eq.~(\ref{equaltriple}) also holds for non-immediate descendant $c$ of $k_2$, hence it cannot be used directly to identify $c$'s location. To overcome this, for each $c$ in $\cal L$, we first check at descendants of $k_2$ before $k_2$. This is ensured by creating a post-order traversal list $\cal W$ \cite{Cormen2001} in Step \ref{postorder}. In post-order traversal $\cal W$, each potential parent is reached only after all its descendants have been reached. The node $k_2$ which satisfies Eq.~(\ref{equaltriple}) in Step \ref{checkstepc} now gives the true parent of leaf $c$. The algorithm terminates after searching for all leaves in $\cal L$.

\textbf{Note:} Empirically computed second moments of voltages may differ from their true values and hence may not satisfy relations from Theorems \ref{theoremcase1} and \ref{theoremcase2} exactly. Hence we use tolerances $\tau_1$ and $\tau_2$ to check the correctness of Eq.~(\ref{equal3}) and Eq.~(\ref{equaltriple}) respectively in Algorithm $1$. We defer the selection of optimal thresholds to future work.

We discuss the operation of Algorithm $1$ through a test example given in Fig.~\ref{fig:example}. As described in the previous paragraph, in the first step, the common parents of leaf node pairs are discovered. In the next step, edges between intermediate nodes are discovered iteratively. Finally, positions of leaf nodes that do not have common parent with any other leaf node are located by a post-order traversal and then the substation connects to the node with no parent.
\begin{figure}[!bt]
\centering
\includegraphics[width=0.4\textwidth,height=.16\textwidth]{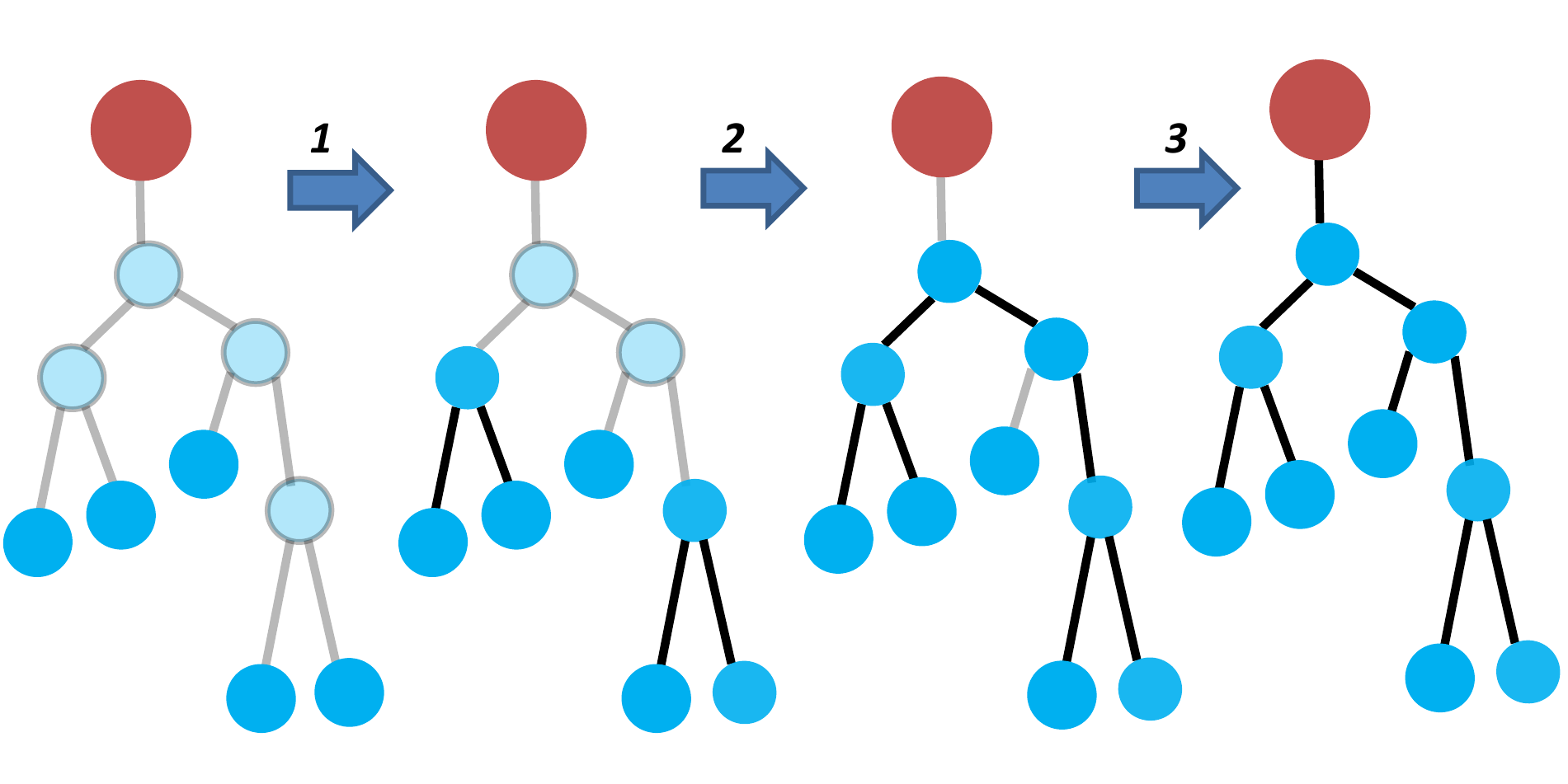}
\squeezeup
\caption{Sequence in which edges area learnt by Algorithm $1$ for an example radial grid using only measurements at leaf nodes. In the first step, leaf pairs with common parent are learnt, then intermediate nodes appear and finally connections to leaf nodes with no leaf sibling are discovered.
\label{fig:example}}
\end{figure}

\section{Computational Complexity and Performance}
\label{sec:complexity}
Consider the original loopy graph $\cal G$ to be complete (all node pairs belong to $\cal E$). We compute the computational complexity of operational tree $\cal $ in terms of the number of nodes $N = |{\cal V}^{\cal T}|$.

\textbf{Algorithm Complexity:} To measure the computational complexity, we focus on the iteration steps. Let the number of leaves in tree $\cal T$ be $l_0$. Detecting edges between leaf pairs and their common parent (Steps \ref{stepa} to \ref{stepaend}) takes $O(l_0^2(N-l_0))$ comparisons, which takes the form $O(N^3)$  in the worst case. Next we analyze the complexity of identifying edges between intermediate nodes. Let the number of `new' intermediate nodes with unknown parents (set ${\cal M}_2$) in each iteration of (Steps \ref{stepb} to \ref{stepbend}) be $l_i$. Each `new' node is compared with the set of all intermediate nodes with unknown parents (${\cal M}_1 \supseteq {\cal M}_2$) first. As addition of a `new' node leads to removal of its children ($\geq 1$) from ${\cal M}_1$, the size of ${\cal M}_1$ never increases more than its initial value ($\leq l_0/2$) giving this step a worst case complexity of $O(l_il_0)$. Comparison of nodes in ${\cal M}_i$ with missing nodes in ${\cal M}- {\cal M}_1$ and leaf nodes in $\cal L$ (Step \ref{stepb2}) has worst case complexity $O(l_il_0(N-l_0))$. Using $l_0$ as $O(N)$ and $\sum l_i = O(N)$, the computational complexity of all iterations in Steps \ref{stepb} to \ref{stepbend} is thus $O(N^3)$. Finally composing the post-order tree traversal and checking at node location ($O(N)$) for each leaf ($O(N)$) has complexity less than or equal to $O(N^2)$. The overall complexity of the algorithm is thus $O(N^3)$. 

\textbf{Maximum Missing Node Fraction:} Consider the two radial graphs given in Fig.~\ref{fig:maxcases} that satisfy the condition in Assumption $2$. The first is a binary tree of depth $d$. All nodes till depth $d-1$ are missing. The second graph has the structure of a line graph with one additional leaf node connected to each intermediate node on the line. In either graph, as the number of nodes increases, the fraction of missing nodes increases to $50\%$. In other words, \emph{configurations exist such that Algorithm $1$ is able to learn the grid structure with approximately half the nodes missing/unobserved.}

\begin{figure}[!bt]
\centering
\hspace*{\fill}
\subfigure[]{\includegraphics[width=0.13\textwidth]{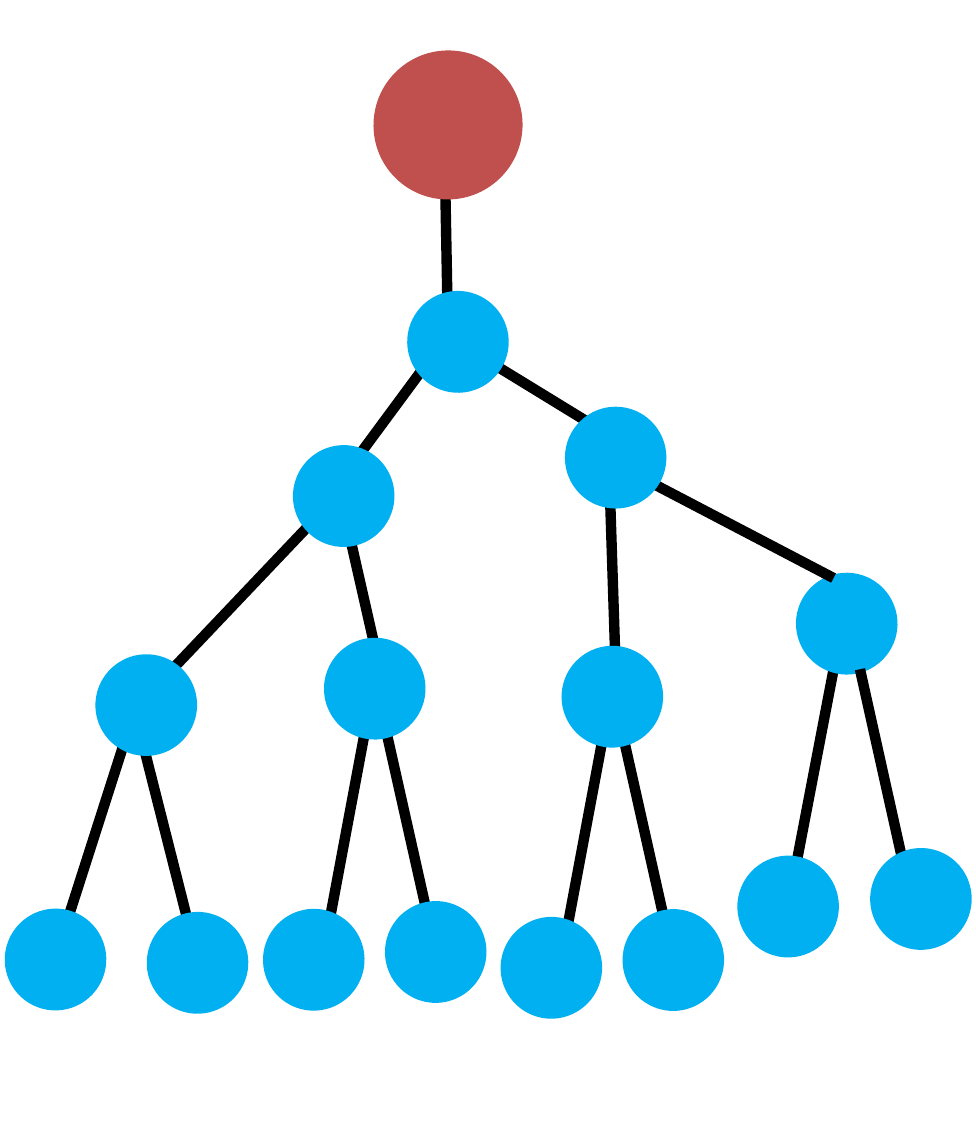}}\hfill
\subfigure[]{\includegraphics[width=0.09\textwidth,height=0.16\textwidth]{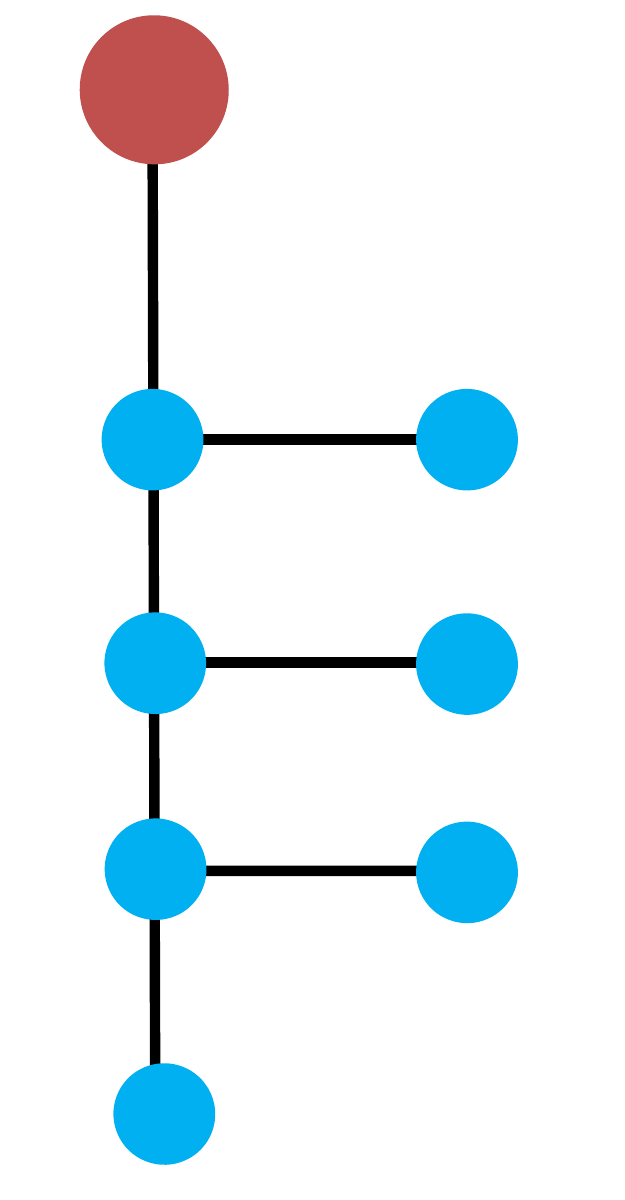}}
\hspace*{\fill}
\squeezeup
\caption{Cases where fraction of missing nodes becomes $50\%$ (a) A binary distribution grid tree with each intermediate node having degree $3$. (b) Distribution grid tree with each intermediate node with one child
\label{fig:maxcases}}
\end{figure}

This is crucial because $50\%$ fraction of missing nodes represents the threshold beyond which state estimation is not possible even in the presence of knowledge of the operational topology. This is because the LC-PF Eqs.~(\ref{LCPF_p}) cannot be solved if less than half the nodes are observed.

\section{Experiments}
\label{sec:experiments}
Here we demonstrate performance of Algorithm $1$. We consider a radial network \cite{testcase2,radialsource} (Fig.~\ref{fig:case}) with $20$ nodes and one substation. Out of those $20$ nodes, $12$ are terminal/leaf nodes while $8$ are intermediate nodes which are missing. In each of our simulation runs, we construct complex power injection samples at the non-substation nodes from a multivariate Gaussian distribution that is uncorrelated between different nodes as per Assumption $1$. Next we use LC-PF Eqs.~(\ref{LC_PF}) to generate nodal voltage magnitude measurements. To understand the performance of our algorithm, we introduce $30$ additional edges (at random) to construct the loopy edge set ${\cal E}$. The additional edges are given impedances comparable to those of operational lines. The input to the observer consists of voltage magnitude measurements and injection statistics at the terminal nodes and impedances of all the lines within the loopy edge set $\cal E$.
The average fraction errors (number of errors/size of the operational edge set) in reconstructing the grid structure is shown in Fig.~\ref{fig:plot} for different sizes of terminal voltage magnitude measurements used. Different curves in the plot depict different values of tolerances $\tau_1$ and $\tau_2$ to check the correctness of Eq.~(\ref{equal3}) and Eq.~(\ref{equaltriple}) in Algorithm $1$. Note that the average fractional errors decreases steadily with increase in the number of measurements as empirical second moment statistics used in Algorithm $1$ get more accurate. The values of tolerance to achieve the most accurate results are selected manually. We plan to develop a theoretical selection of the optimal tolerance values in future work.

\begin{figure}[!bt]
\centering
\includegraphics[width=0.356\textwidth,height =.16\textwidth]{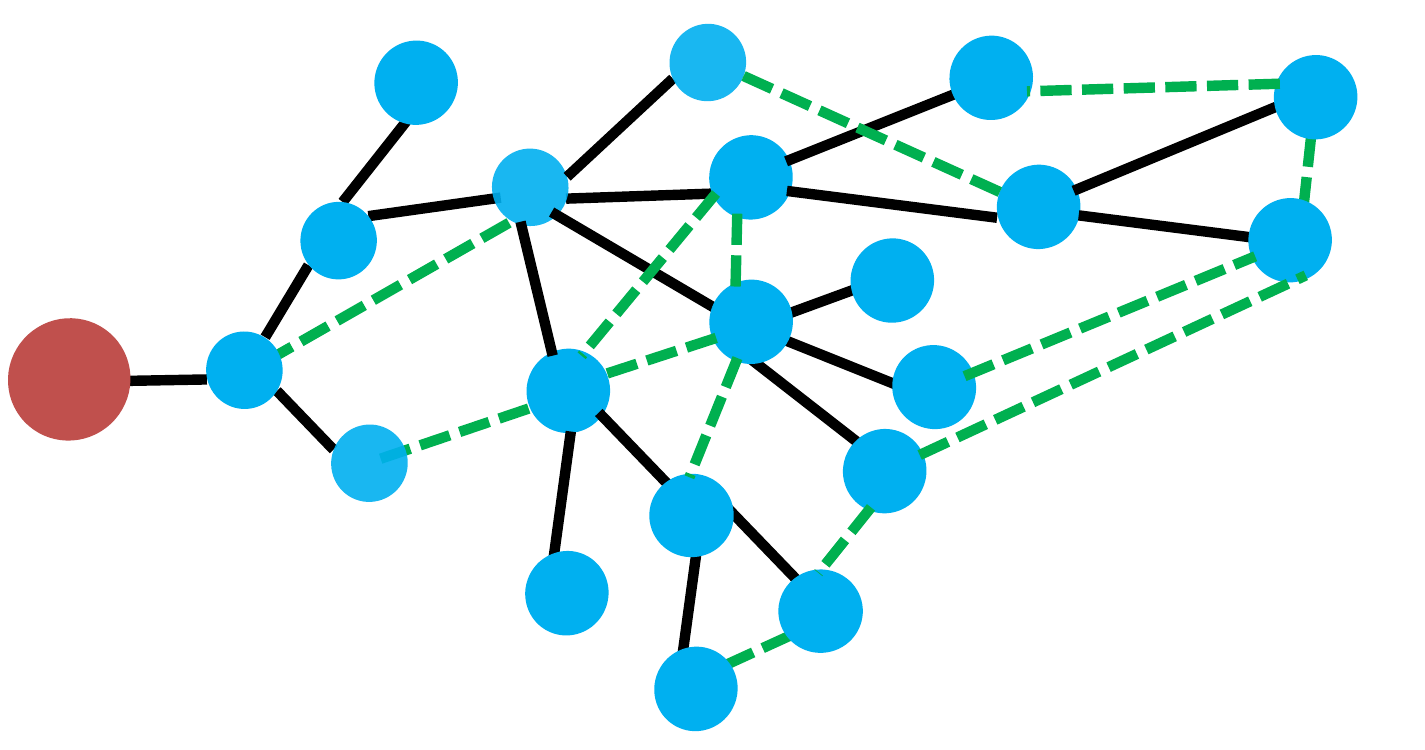}
\vspace{-.10cm}
\caption{Layouts of the grid tested with $20$ non-substation nodes. Black lines represent operational edges. Some of the additional open lines (actual number $30$) are represented by dotted green lines.}
\label{fig:case}
\vspace{-2mm}
\end{figure}

\begin{figure}[!bt]
\centering
\subfigure[]{\includegraphics[width=0.39\textwidth,height = .35\textwidth]{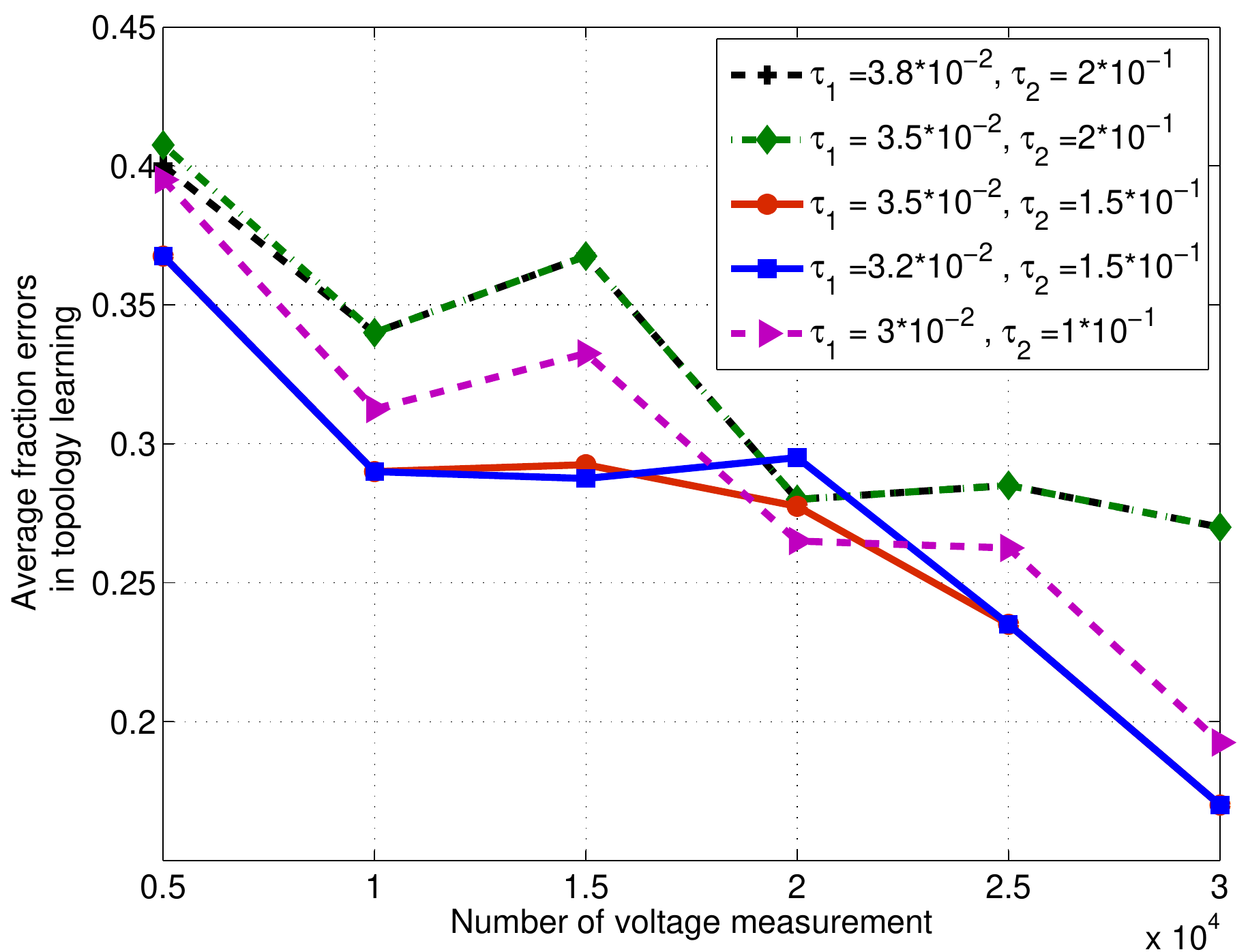}\label{fig:plot}}
\squeezeup
\caption{Average fractional errors in learning operational edges vs number of samples used in Algorithm $1$ for different values of tolerances $\tau_1$ and $\tau_2$.
\label{fig:algo2}}
\end{figure}

\section{Conclusions}
\label{sec:conclusions}
Topology learning is an important problem in distribution grids as it influences several other control and managements applications that form part of the smart grid. The learning problem is complicated by the low placement of real-time measurements in the distribution grid, primarily being placed at terminal nodes that represent end-users or households. In this paper, we study the problem of learning the operational radial structure from a dense underlying loopy grid graph in this realistic regime where only terminal node (end-user) voltage measurements and injection statistics are available as input and all other nodes are unobserved/missing. We use a linear-coupled power flow model and show that voltages terminal node pairs and triplets satisfy relations that depend on their individual injection statistics and impedances on the lines connecting them. We use these properties to propose a learning algorithm that iteratively learns the grid structure from the leaves onward towards the substation root. We show that the algorithm has polynomial time complexity comparable to existing work, despite being capable of tolerating much greater and realistic fraction of missing nodes. For specific cases, this algorithm is capable of learning the structure with $50\%$ unobserved nodes, beyond which state estimation is not possible even in the presence of topology information. We demonstrate performance of the learning algorithm through experiments on distribution grid test case. Computing the sample complexity of the learning algorithm and optimizing the selection of tolerance values based on the number of samples are directions of our future research in this area. Further, we plan to expand this learning regime to cases with correlated injections.
\bibliographystyle{IEEEtran}
\bibliography{FIDVR,SmartGrid,voltage,trees}
\end{document}